\documentclass{amsart}
\usepackage{tikz}
\usepackage{xcolor}
\usepackage{amssymb,latexsym,amsmath,extarrows}
\usepackage{amsthm}
\usepackage{mathabx}
\usepackage{graphicx,mathrsfs,comment}
\usepackage{hyperref,url}
\usepackage{pict2e}
\usepackage{enumerate}
\usepackage{bm}

\usepackage{cancel}

\usepackage{amstext}
\usepackage{bbm} 

\numberwithin{equation}{section}

\usepackage{esint}

\newtheorem{theorem}{Theorem}[section]
\newtheorem{lemma}[theorem]{Lemma}

\newtheorem{proposition}[theorem]{Proposition}

\newtheorem*{remark*}{Remark}

\newtheorem{definition}[theorem]{Definition}
\newtheorem{corollary}[theorem]{Corollary}

\makeatletter
\newcommand{\barredsum}{%
  \DOTSB\mathop{\mathpalette\@barredsum\relax}\slimits@
}
\newcommand{\@barredsum}[2]{%
  \begingroup
  \sbox\z@{$#1\sum$}%
  \setlength{\unitlength}{\dimexpr2pt+\ht\z@+\dp\z@\relax}%
  \@barredsumthickness{#1}%
  \vphantom{\@barredsumbar}%
  \ooalign{$\m@th#1\sum$\cr\hidewidth$#1\@barredsumbar$\hidewidth\cr}%
  \endgroup
}
\newcommand{\@barredsumbar}{%
  \vcenter{\hbox{\begin{picture}(0,1)\roundcap\Line(0,0)(0,1)\end{picture}}}%
}
\newcommand{\@barredsumthickness}[1]{
  \linethickness{%
    1.25\fontdimen8
      \ifx#1\displaystyle\textfont\else
      \ifx#1\textstyle\textfont\else
      \ifx#1\scriptstyle\scriptfont\else
      \scriptscriptfont\fi\fi\fi 3
  }%
}
\makeatother

\newcommand{\ga}{\gamma}

\newcommand{\de}{\delta}

\newcommand{\e}{\varepsilon}

\newcommand{\la}{\lambda}

\newcommand{\si}{\sigma}
\newcommand{\Si}{\Sigma}
\newcommand{\vp}{\varphi}

\newcommand{\cv}{\mathcal V}

\newcommand{\cq}{\mathcal Q}

\newcommand{\wh}{\widehat}

\newcommand{\ZR}{\mathbb{R}}
\newcommand{\ZT}{\mathbb{T}}

\newcommand{\ZS}{\mathbb{S}}

\newcommand{\Id}{{\bf{1}}}

\newcommand{\cT}{{\mathcal T}}

\newcommand{\BR}{{\rm{Br}}}

\newcommand{\supp}{{\rm supp}}

\begin{document}

\title[A weighted restriction estimate in $\mathbb R^2$]{A weighted restriction estimate in $\mathbb R^2$}

\date{}

\author{Zhenbin Cao, Feilong Guo, Junfeng Li and Xiangyu Wang} 
\address{Institute of Mathematics, Henan Academy of Sciences, Zhengzhou 450046, China}
\email{caozhenbin@hnas.ac.cn}
\address{School of Mathematical Sciences, Dalian University of Technology, Dalian, 116024, China} \email{feilongguo2022@163.com}
\address{School of Mathematical Sciences, Dalian University of Technology, Dalian, 116024, China} \email{junfengli@dlut.edu.cn}
\address{Department of Mathematics, University of Illinois Urbana-Champaign, Urbana, IL 61801, USA}
\email{xw70@illinois.edu}

\begin{abstract}

We establish a weighted restriction estimate in $\mathbb{R}^2$. Our proof combines new two-ends Furstenberg estimates with refined decoupling. As an application, our result yields decay estimates for the circular $L^p$-means of the Fourier transforms of fractal measures in $\mathbb{R}^2$.

\end{abstract}
\maketitle


\section{Introduction}
The main objective of this paper is to extend recent weighted $L^2$ estimates \cite{Wu} for the Fourier extension operator. As an application of these estimates, we obtain new bounds for the circular $L^p$-means of the Fourier transforms of fractal measures in $\mathbb{R}^2$.

\subsection{Overview}

We first give a brief introduction to the weighted $L^2$ estimate considered in this paper.
Let $S\subset \ZR^2$ be a compact curve with nonvanishing curvature. Let $\si_S$ denote the surface measure on $S$.
The Fourier extension operator associated with $S$ is defined by
\[    E_S f(x)
    :=
    \int_S e^{2\pi i x\cdot\xi} f(\xi)\,d\si_S(\xi),
    \qquad f:\mathbb{R}^2  \to \mathbb{C}, \quad x\in\ZR^2.\]
Let $B_r(y)\subset \mathbb R^2$ denote the ball of radius $r$ centered at $y$.
Fix $t\in[0,2]$ and $0<\de\ll 1$, and let $\mathcal{B}_\de$ be a finitely
overlapping collection of balls of radius $\de$ covering $B_1(0)$.
For a set $E\subset B_1(0)$, we say that $E$ is a {\bf Katz-Tao $(\de,t)$-set} if
\begin{equation*}
    \# \{B_\de \in \mathcal{B}_\de: E\cap B(x,r) \cap B_\de \neq \emptyset \}\lesssim (r/\de)^t, \hspace{.3cm}\forall x \in\ZR^2, \,r\in[\de,1].
\end{equation*}
We state our main result as follows.
\begin{theorem}\label{main-result}
Suppose
\begin{equation}\label{t-full-range}
0 \leq t \leq 2.
\end{equation}
Let $X\subset B_R(0)$ be a union of unit balls with $R\gg1$. Assume that the $R^{-1}$-dilate of $X$ is a Katz--Tao
$(R^{-1},t)$-set. Then for all $f:\mathbb{R}^2 \to \mathbb{C}$, the estimate
\[
\|E_Sf\|_{L^q(X)}\lesssim_{\e,p} R^\e\|f\|_{L^p(S,d\si_S)}
\]
holds, provided that
\begin{equation} \label{pq-range-final}
    p \geq 2, \qquad q \geq \max\left\{\max\left\{\frac{6t+12}{6-t}, \frac{18t}{t+4}\right\},\min\{2t,t+1\}\left(1-\frac{1}{p}\right)^{-1},p\right\}.
\end{equation}
\end{theorem}
As a consequence, we obtain new bounds for the circular $L^p$-means of the
Fourier transforms of fractal measures in $\mathbb{R}^2$.
\begin{theorem}\label{decay}
Let $t$ satisfy \eqref{t-full-range}, and let $\mu$ be a probability
measure supported in $B_1(0)$ satisfying the Frostman condition
\[
\mu(B_\rho(x))\lesssim \rho^t
\]
for every $x\in\mathbb{R}^2$ and $0<\rho\leq1$.
Then, for every $1\leq p\leq2$ and $\e>0$,
\[
\left(
\int_{\ZS^1}|\widehat{\mu}(R\xi)|^p\,d\sigma(\xi)
\right)^{1/p}
\lesssim_{\e,p}
R^{-\beta_t(p)+\e},
\]
where
\[
\beta_t(p)
:=
\begin{cases}
\displaystyle
\min\left\{
\frac{t(6-t)}{6(t+2)},\frac{1}{2p}
\right\},
& 0\leq t<1,\\[8pt]
\displaystyle
\min\left\{
\frac{t+4}{18},\frac{t}{(t+1)p}
\right\},
& 1\leq t\leq2.
\end{cases}
\]
\end{theorem}

Compared with the bounds obtained from the classical circular
$L^2$-mean estimate and its consequences by H\"older's inequality and
interpolation, Theorem \ref{decay} yields a genuine improvement on both
sides of $t=1$. More precisely, for
\[
\frac{9-\sqrt{33}}{4}<t<1,
\]
our estimate gives a strictly stronger decay rate for every
$1\leq p<2$. In particular, in the range
\[
\frac{3t+6}{-t^2+6t}\leq p<2,
\]
we obtain the decay rate
\[
R^{-1/(2p)+\e}.
\]
Similarly, for
\[
1<t<\frac{8}{7},
\]
the endpoint estimate furnished by Theorem \ref{decay}, together with
interpolation with the classical circular $L^2$-mean bound, yields a
strict improvement for every $1\leq p<2$. The value $t=8/7$ is exactly
the threshold at which the new endpoint decay rate coincides with the
classical $L^2$ decay rate. Moreover, for
\[
\frac{9-\sqrt{33}}{4}\leq t<1
\qquad\text{and}\qquad
\frac{3t+6}{-t^2+6t}\leq p\leq2,
\]
the exponent $1/(2p)$ is sharp up to the $R^\e$ loss.

The case $t=1$ was established by Wu \cite{Wu} by combining refined
decoupling \cite{GIOW} with a two-ends Furstenberg estimate \cite{Wang-Wu24}.
In this paper, we extend this result to the range \eqref{t-full-range}
using a new two-ends Furstenberg estimate \cite{Wang-Wu26}.

Weighted restriction estimates have been extensively studied in Fourier
analysis, in part because of their close connections with problems in
geometric measure theory. By incorporating weights or measures with
prescribed dimensional properties, these estimates provide a natural
framework for studying the interaction between Fourier extension operators
and fractal sets. They have found important applications to the decay of
spherical and circular averages of Fourier transforms of fractal measures
and, consequently, to the Falconer distance set problem; see, for instance,
\cite{Wolff, Erdogan, Shayya, DGOWWZ, LucaRogers}. These connections have
motivated substantial developments in weighted and fractal restriction
theory over the past several decades.
\bigskip

\subsection{Structure of the article}

The paper is organized as follows. In Section \ref{section-preliminary},
we collect several preliminary reductions and technical tools, including
refined decoupling, two-ends Furstenberg estimates, and the broad-narrow argument. In Section \ref{proof-main-result}, we establish the $L^2$-broad estimates, which constitute the main result of this paper.
Section \ref{proof-decay} is devoted to the proof of Theorem \ref{decay}.
Finally, Section \ref{proof-example} contains a further discussion of
$L^2$-broad estimates.

\bigskip

\subsection{Notation}
We use the Fourier convention $e(z):=e^{2\pi i z}$.
For a finite set $E$, we denote its cardinality by $\#E$. If
$\mathcal{E}$ is a collection of subsets of $\mathbb{R}^n$, we denote
the union of all sets in $\mathcal{E}$ by $\bigcup \mathcal{E}$.

We write $A\lesssim B$, or equivalently $B\gtrsim A$, if
\[
    |A|\leq C|B|
\]
for some constant $C>0$. When the implicit constant depends on a
parameter $L$, we write $A\lesssim_L B$. We use $A\ll B$, or
equivalently $B\gg A$, when the implicit constant may be chosen
arbitrarily small. The notation $A\lessapprox B$ (equivalently,
$B\gtrapprox A$) means that
\[
    A\lesssim_\eta R^\eta B
\]
for every $\eta>0$.

For $r>1$, we use $\mathrm{RapDec}(r)$ to denote a quantity that
decays rapidly in $r$; more precisely,
\[
    \mathrm{RapDec}(r)\lesssim_N r^{-N}
\]
for every integer $N>0$. Finally, for a set $X$, we denote its
$d$-neighborhood by $N_d(X)$.

\medskip

\subsection{Acknowledgments}
X. Wang is deeply grateful to Professor Shukun Wu for suggesting this research topic.
The authors acknowledge the use of ChatGPT for language editing, LaTeX formatting,
proofreading for typographical and citation errors, and assistance in identifying
and addressing gaps in the mathematical justification in Section~\ref{proof-example},
in particular in the construction establishing the broad lower bound.

\bigskip

\section{Preliminaries}\label{section-preliminary}
In this section, we collect several preliminary reductions and technical
tools that will be used throughout the paper. These include refined
decoupling, two-ends Furstenberg estimates, and the broad-narrow argument.
We also record some auxiliary estimates and notation that will be useful
in the proof of our main result.

A standard reduction shows that Theorem \ref{main-result} is equivalent to the following statement.
\begin{theorem}\label{main-result-rewritten}
Suppose that $t$ satisfies \eqref{t-full-range}. For $f:\mathbb{R}\to\mathbb{C}$ and a real-valued phase $\Phi$ on $[-1,1]$, define the extension operator
$E_\Phi f:\mathbb{R}^2\to\mathbb{C}$ by
\[
E_\Phi f(x_1,x_2)
:=
\int_{-1}^{1}
e^{2\pi i(x_1\xi+x_2\Phi(\xi))}
f(\xi)\,d\xi.
\]
Let $X\subset B_R(0)$ with $R\gg 1$. Assume that the $R^{-1}$-dilate
of $X$ is a Katz--Tao $(R^{-1},t)$-set. Then, for every
$f\in\mathcal{S}(\mathbb{R})$ with $\operatorname{supp} f\subset[-1,1]$ and $\Phi\in C^\infty([-1,1];\mathbb{R})$ satisfying
\begin{equation}\label{phi-condition}
    \Phi(0) = \Phi'(0) = 0\qquad \text{and}\qquad \Phi''(\xi)\sim 1\quad \text{for all }\xi\in[-1,1],
\end{equation}
we have
\[
\|E_\Phi f\|_{L^q(X)}
\lesssim_{\e,p}
R^\e \|f\|_p,
\]
provided that \eqref{pq-range-final} holds.
\end{theorem}
Throughout the remainder of the paper, we assume that
$f\in\mathcal{S}(\mathbb{R})$ with
$\operatorname{supp} f\subset[-1,1]$, that
$\Phi\in C^\infty([-1,1];\mathbb{R})$ satisfies
\eqref{phi-condition}, and that $X$ is a union of unit balls.
\subsection{Wave Packet Decomposition and Refined Decoupling}
\medskip
We now recall the wave packet decomposition at scale $r \gg 1$ and record
the properties that will be needed later. Since this construction is
standard, we only describe the relevant setup.
Fix $\de_0:=\e^{1000}$.
For this decomposition, we extend $\Phi$ smoothly to a fixed neighborhood
of $[-1,1]$ and interpret $E_\Phi h$ as an integral over the full
support of $h$ whenever $h$ is a wave packet or a finite sum of wave
packets. This agrees with the original operator on functions supported
in $[-1,1]$.

Let $\Theta$ be a finitely overlapping family of intervals of length
$r^{-1/2}$ covering $[-1,1]$. Choose a smooth partition of unity
$\{\vp_\theta\}_{\theta\in\Theta}$ subordinate to the enlarged intervals
$\{2\theta\}_{\theta\in\Theta}$, so that
\[
\supp(\vp_\theta)\subset 2\theta,
\qquad
\sum_{\theta\in\Theta}\vp_\theta(\xi)=1
\quad\text{for }\xi\in[-1,1].
\]
For a function $f$ supported in $[-1,1]$, set
\[
f_\theta:=f\vp_\theta.
\]

To further localize each $f_\theta$, let $\cv$ be a finitely overlapping
cover of $\mathbb{R}$ by intervals $v$ of length $r^{1/2}$.
For each $v\in\cv$, choose a smooth function $\psi_v$ adapted to $v$
such that
\[
\supp(\wh{\psi}_v)
\subset [-r^{-1/2},r^{-1/2}]
\]
and
\[
\sum_{v\in\cv}\psi_v=1
\qquad\text{on }\mathbb{R}.
\]
We then decompose
\begin{equation}
\label{wave-packet-decomposition}
f
=
\sum_{\theta\in\Theta}
\sum_{v\in\cv}
(f\vp_\theta)\ast\wh{\psi}_v
=
\sum_{(\theta,v)\in\Theta\times\cv}
f_{\theta,v},
\end{equation}
where
\[
f_{\theta,v}
:=
(f\vp_\theta)\ast\wh{\psi}_v.
\]

Each pair $(\theta,v)$ is associated with a tube in the physical space.
Let $c_\theta$ and $c_v$ denote the centers of $\theta$ and $v$,
respectively, and define
\[
T_{\theta,v}
:=
\left\{
(x_1,x_2)\in B_r(0):
\left|
x_1-c_v+x_2\Phi'(c_\theta)
\right|
\leq r^{1/2+\de_0}
\right\}.
\]
Thus $T_{\theta,v}$ has dimensions
$r^{1/2+\de_0}\times r$ and is oriented according to the frequency
interval $\theta$. We write
\[
V(\theta):=(\Phi'(c_\theta),-1)
\]
for the corresponding direction vector.

For a fixed $\theta\in\Theta$, let
\[
{\ZT}(\theta)
:=
\left\{
T_{\theta,v}:
v\in\cv,\;
T_{\theta,v}\cap B_r(0)\neq\varnothing
\right\},
\]
and set
\[
{\ZT}
:=
\bigcup_{\theta\in\Theta}{\ZT}(\theta).
\]
Whenever $T=T_{\theta,v}\in {\ZT}$, we use the shorthand
\[
f_T:=f_{\theta,v},
\qquad
\theta_T:=\theta.
\]
The wave packet decomposition immediately yields the following properties.
\begin{lemma} \label{wpt}
	The wave packet decomposition satisfies the following properties.
	\begin{enumerate}\item On $B_r(0)$,
        $E_\Phi f=\sum_{T\in\ZT}E_\Phi f_T
        +O(\mathrm{RapDec}(r)\|f\|_2)$.
		\item $\supp f_{T}\subset 4\theta$ when $T$ has direction $V(\theta)$.
		\item $\{V(\theta)\}_{\theta\in\Theta}$ are $r^{-1/2}$-separated.
	\end{enumerate}    
\end{lemma}
It is standard that each wave packet $E_\Phi f_T$ is essentially
concentrated on $T$. We also recall the
$L^2$ almost-orthogonality of the wave packet decomposition.

\begin{lemma}\label{spatial-concentration}
Let $T\in\mathbb{T}$. For every $x\in B_r(0)\setminus T$, we have
\[
    |E_\Phi f_T(x)|
    \lesssim
    \mathrm{RapDec}(r)\,\|f\|_2.
\]
\end{lemma}

\begin{lemma}\label{l2-orthogonality}
Let $\mathbb{T}'\subset\mathbb{T}$. Then
\[
    \left\|
        \sum_{T\in\mathbb{T}'} f_T
    \right\|_2^2
    \sim
    \sum_{T\in\mathbb{T}'} \|f_T\|_2^2.
\]
\end{lemma}
Decoupling theory provides a powerful framework for controlling the
$L^p$ norm of an oscillatory integral by decomposing its frequency
support into smaller pieces and estimating the contributions of the
corresponding wave packets. In the present setting, we decompose the
frequency interval into caps of length $r^{-1/2}$ and apply decoupling
to relate the global behavior of $E_\Phi f$ to that of the localized
pieces $E_\Phi f_\theta$. Refined decoupling further incorporates
information on the spatial distribution of the associated wave packets,
and therefore yields stronger estimates when the wave packets satisfy
additional concentration or multiplicity conditions. The refined
decoupling inequality proved in \cite[Theorem 4.2]{GIOW} may be viewed as a
strengthening of the classical decoupling inequality.
\begin{theorem}
\label{refined-decoupling-thm}
Let $f\in\mathcal{S}(\mathbb{R})$ with $\operatorname{supp} f\subset[-1,1]$. Suppose that $f$ is a sum of wave packets $f=\sum_{T\in\ZT}f_T$ such that the quantities $\|f_T\|_2$ are comparable for all $T\in\ZT$. 
Let $X$ be a union of $r^{1/2}$-balls in $B_r(0)$ such that each $r^{1/2}$-ball $Q\subset X$ intersects at most $M$ tubes from $\ZT$. 
Then 
\begin{equation}
\label{refined-decoupling}
    \|E_\Phi f\|_{L^6(X)}\lessapprox  M^{1/3}\Big(\sum_{T\in\ZT}\|E_\Phi f_T\|_{L^6(w_{B_r})}^6\Big)^{1/6}.
\end{equation}
Here $w_{B_r}$ is a weight that is $\sim1$ on $B_r(0)$ and decreases rapidly outside $B_r(0)$. 
\end{theorem}

\subsection{Two-ends Furstenberg Estimates}
In this subsection, we introduce a recent two-ends Furstenberg estimate
established in \cite{Wang-Wu26}. For our purposes, we recast the result in a
form suited to the present setting and state the precise version needed
in the proof of our main result.

We begin by recalling the two-ends condition. 
\begin{definition}
Let $L$ be a family of lines in $\ZR^2$ and let $\de\in(0,1)$.
A {\bf shading} is an assignment $\ell\mapsto Y(\ell)$ such that $Y(\ell)\subset N_\de(\ell)\cap B_1(0)$ is a union of finitely overlapping $\de$-balls in $\ZR^2$ for all $\ell\in L$.
We say $Y$ is {\bf $\la$-dense} if $|Y(\ell)|\geq \la |N_\de(\ell)\cap B_1(0)|$ for every $\ell\in L$.
\end{definition}

\begin{definition}
Let $\de\in(0,1)$ and let $(L,Y)_\delta$ be a set of lines and shading.
Let $0<\e_2<\e_1<1$.
We say $Y$ is {\bf $(\e_1 ,\e_2)$-two-ends} if there exists a constant $C$, independent of $\de$, $\ell$, and $J$, such that for every $\ell\in L$ and every $\de\times\de^{\e_1}$-tube $J\subset N_\de(\ell)$,
\begin{equation}
\nonumber
    |Y(\ell)\cap J|\leq C\de^{\e_2} |Y(\ell)|.
\end{equation}
\end{definition}

We next recall the two-ends Furstenberg estimate from
\cite[Theorem 1.5]{Wang-Wu26}, together with its elementary endpoint cases.

\begin{definition}
\label{gamma-Y-def}
Let $\de\in(0,1]$, and let $(L,Y)_\de$ be a set of lines and shading.
Let $t\in[0,1]$.
For each $\ell\in L$, define 
\[
\ga_{Y,t}(\ell):=\sup_{\substack{r\in[\de,1]\\ x\in\ell}}\left(\frac{\de}{r}\right)^t(\de^{-2}|B_r(x)\cap Y(\ell)|),
\]
and define $\ga_{Y,t}:=\sup_{\ell\in L}\ga_{Y,t}(\ell)$.
\end{definition}

\begin{definition}
A set of lines in $\ZR^2$ is a Katz--Tao $(\de,t)$-set if its dual point set in $\ZR^2$ is a Katz--Tao $(\de,t)$-set.
\end{definition}

\begin{theorem}
\label{two-ends-furstenberg-general-intro}
Let $(L,Y)_\de$ be a set of $\de$-separated lines in $\ZR^2$ with a $(\e_1, \e_2)$-two-ends shading  such that $|Y(\ell)|\geq\la\de$ for all $\ell\in L$.
Let $t\in[0,2]$, and $t^\ast=\min\{t, 2-t\}$. 
Suppose the set of lines $L$ is a Katz--Tao $(\de, t)$-set. Then
\[
    \Big|\bigcup_{\ell\in L}Y(\ell)\Big|\gtrapprox_{\e_2} \de^{t\e_1/2} \la^{1/2}\de^{(t-1)/2}\ga_{Y,t^\ast}^{-1/2}\sum_{\ell\in L}|Y(\ell)|.
\]
\end{theorem}
\begin{proof}
For $0<t<2$, this is \cite[Theorem 1.5]{Wang-Wu26}.
At either endpoint, $t^\ast=0$ and
$\ga_{Y,0}\gtrsim\la/\de$.
If $t=0$, the Katz--Tao condition gives $\#L\lesssim1$, so the
conclusion follows by bounded overlap.
If $t=2$, a point belongs to at most $O(\de^{-1})$ of the
$\de$-neighborhoods of the $\de$-separated lines. Hence
$|\bigcup_{\ell\in L}Y(\ell)|\gtrsim
\de\sum_{\ell\in L}|Y(\ell)|$, which implies the stated bound.
\end{proof}

We next recast this result in the form needed for our argument.

\begin{lemma}\label{two-ends-furstenberg-estimate}
Let $(L,Y)_\de$ be a collection of
$\de$-separated lines in $\ZR^2$ equipped with an
$(\e_1,\e_2)$-two-ends shading. Let $1 \leq M\leq \de^{-1+\e_1}$. Assume that
\[
|Y(\ell)|\geq \la\de
\]
for every $\ell\in L$. Let $t\in[0,2]$ and suppose that $L$ is a
Katz--Tao $(\de,t)$-set. Furthermore, assume that for every
$\ell\in L$ and $x \in \ell$,
\[
\de^{-2}|B_{\de^{\e_1}}(x)\cap Y(\ell)|\lesssim M.
\]
Then,
\[
\left|\bigcup_{\ell\in L}Y(\ell)\right|
\gtrapprox_{\e_2}
\de^{O(\e_1)}
\la^{1/2}
\de^{(t-1)/2}
M^{-\frac12\max\{t-1,1-t\}}
\sum_{\ell\in L}|Y(\ell)|.
\]
\end{lemma}
\begin{proof}
Let
\[
t^\ast:=\min\{t,2-t\}.
\]
By Theorem \ref{two-ends-furstenberg-general-intro}, it suffices to prove that
\begin{equation}\label{ga0}
    \ga_{Y,t^\ast}
    \lesssim
    (\de^{-\e_1}M)^{\max\{t-1,1-t\}}.
\end{equation}

Fix $\ell\in L$. For $r\in[\de^{\e_1},1]$, we have
\[
\sup_{x\in\ell}
\left(\frac{\de}{r}\right)^{t^\ast}
\de^{-2}|B_r(x)\cap Y(\ell)|
\lesssim
\de^{t^\ast}\de^{-\e_1}M.
\]
Since
\[
(\de^{1-\e_1} M)^{t^\ast}\lesssim 1,
\]
it follows that
\begin{equation}\label{ga1}
\sup_{\substack{r\in[\de^{\e_1},1]\\ x\in\ell}}
\left(\frac{\de}{r}\right)^{t^\ast}
\de^{-2}|B_r(x)\cap Y(\ell)|
\lesssim
(\de^{-\e_1}M)^{\max\{t-1,1-t\}}.
\end{equation}

On the other hand, 
\[
\sup_{\substack{r\in[\de,\de^{\e_1}]\\ x\in\ell}}
\left(\frac{\de}{r}\right)^{t^\ast}
\de^{-2}|B_r(x)\cap Y(\ell)|
\lesssim
M^{\max\{t-1,1-t\}}.
\]
Combining this estimate with \eqref{ga1} yields \eqref{ga0}.
\end{proof}
By point-line duality, one has
\begin{corollary}\label{two-ends-furstenberg-estimate-dual}
Let $\cq$ be a collection of $\de$-balls contained in $B_1(0)$.
For each $Q\in\cq$, let $\cT(Q)$ be a family of
$\de\times 1$ rectangles intersecting $Q$.
Let
\[
0<\e_2<\e_1<1,
\qquad
M\geq 1.
\]
For an arc $\si\subset\ZS^1$ of length $\de^{\e_1}$, define
\[
\cT_\si(Q)
:=
\left\{
T\in\cT(Q):
\operatorname{dir}(T)\in\si
\right\}.
\]

Suppose that $t\in[0,2]$ and that $\cq$ is a Katz--Tao
$(\de,t)$-set. Assume further that, for every $Q\in\cq$,
\[
\#\cT(Q)\geq M,
\]
and that, for every arc $\si\subset\ZS^1$ of length $\de^{\e_1}$,
\[
\#\cT_\si(Q)\sim M
\qquad\text{or}\qquad
\#\cT_\si(Q)=0,
\]
together with
\[
\#\cT_\si(Q)
\lesssim
\de^{\e_2}\#\cT(Q).
\]
Then
\[
\#\bigcup_{Q\in\cq}\cT(Q)
\gtrapprox_{\e_2}
\de^{O(\e_1)}\de^{t/2}\#\cq
\begin{cases}
M^{1+t/2},
& 0\leq t<1,\\[4pt]
M^{2-t/2},
& 1\leq t\leq2.
\end{cases}
\]
\end{corollary}

\medskip

\subsection{Broad--Narrow Argument}
In this subsection, we reduce Theorem \ref{main-result-rewritten} to an $L^2$-broad estimate. We first recall the definition of the broad norm, which will be used throughout the decoupling argument. Similar formulations may be found
in \cite{BG,G,Wu}.

\begin{definition} \label{definition-broad}
Let $K \sim R^{\varepsilon^4}$. Let
$\Si=\{\si\}$ be a finitely overlapping collection of $K^{-1}$-intervals
covering $[-1,1]$. For each $\si\in\Si$, set
\[
    f_\si:=f\Id_\si.
\]
For $x\in\ZR^2$ and integer $A\geq 1$, define the $A$-broad part of $E_\Phi f$ by
\[
    \BR_A E_\Phi f(x)
    :=
    \max_{\substack{\Si'\subset\Si\\ \#\Si'=A}}
    \min_{\si\in\Si'}
    |E_\Phi f_\si(x)|.
\]
\end{definition}
The broad norm satisfies the following standard triangle inequality;
see \cite[Lemma 2.3]{Wu}.

\begin{lemma}\label{triangle-inequality}
Suppose that $A=A_1+A_2$ and $f=f_1+f_2$. Let $q \geq 2$ and $U \subset \mathbb{R}^2$. Then
\[
    \|\BR_A E_\Phi f\|_{L^q(U)}
    \leq
    \|\BR_{A_1} E_\Phi f_1\|_{L^q(U)}
    +
    \|\BR_{A_2} E_\Phi f_2\|_{L^q(U)}.
\]
\end{lemma}
We now state the main result of this paper.

\begin{theorem}
Let $0\leq t\leq 2$, and let $X\subset B_R(0)$ be a union of unit balls. Suppose $A\sim R^{\e^{100}}$. 
Assume that the $R^{-1}$-dilate of $X$ is a Katz--Tao
$(R^{-1},t)$-set. Then, for every
$f\in\mathcal{S}(\mathbb{R})$ with
$\operatorname{supp} f\subset[-1,1]$ and every $\Phi\in C^\infty([-1,1];\mathbb{R})$ satisfying
\eqref{phi-condition}, we have
\[
\|\BR_A E_\Phi f\|_{L^2(X)}
\lesssim_\e
R^\e |X|^{\frac{2t}{3t+6}}\|f\|_2,
\qquad \text{if }0\leq t<1,
\]
and
\[
\|\BR_A E_\Phi f\|_{L^2(X)}
\lesssim_\e
R^\e |X|^{\frac{4t-2}{9t}}\|f\|_2,
\qquad \text{if } 1\leq t\leq2.
\]
\end{theorem}
By pigeonholing, we have
\begin{corollary}\label{l2-lq-broad-estimate}
Let $0\leq t\leq 2$ and $X\subset B_R(0)$. Suppose $A\sim R^{\e^{100}}$. 
Assume that the $R^{-1}$-dilate of $X$ is a Katz--Tao
$(R^{-1},t)$-set. Then, for every
$f\in\mathcal{S}(\mathbb{R})$ with
$\operatorname{supp} f\subset[-1,1]$ and every
$\Phi\in C^\infty([-1,1];\mathbb{R})$ satisfying
\eqref{phi-condition}, we have
\[
\|\BR_A E_\Phi f\|_{L^q(X)}
\lesssim_\e
R^\e \|f\|_2,
\]
for 
\begin{equation} \label{small-t-q-range}
    q \geq \frac{6t+12}{6-t},
\qquad \text{if }0\leq t<1,
\end{equation}
and
\begin{equation} \label{large-t-q-range}
    q \geq \frac{18t}{t+4},
\qquad \text{if }1\leq t\leq2.
\end{equation}
\end{corollary}

The following lemma allows us to deduce Theorem \ref{main-result-rewritten} from the preceding estimate.
It follows from the random sampling argument in
\cite[proof of Lemma 1.6]{Wang-Wu24}; see also
\cite[Lemma 2.7]{Wu} for the case $t=1$.

\begin{lemma}[Katz--Tao extraction]
\label{katz-tao-set-lem}
Let $\de\in(0,1)$ and $0\leq t\leq2$. Suppose that
$X\subset B_1(0)$ is a union of $\de$-balls, and define
$\ga\geq1$ by
\[
\ga
:=
\sup_{\substack{\de\leq r\leq1\\ x\in\mathbb{R}^2}}
\frac{|X\cap B(x,r)|\,\de^{-2}}
{(r/\de)^t}.
\]
Then there exists a subset $X'\subset X$ satisfying
\[
|X'|\approx \ga^{-1}|X|
\]
such that $X'$ is a Katz--Tao $(\de,t)$-set.
\end{lemma}
\begin{proof}[Proof of Theorem \ref{main-result-rewritten} via broad estimates]
The case $q=\infty$ follows from
$\|E_\Phi f\|_\infty\leq\|f\|_1\lesssim_p\|f\|_p$.
Hence we may assume that $2\leq p\leq q<\infty$.
By Definition \ref{definition-broad}, we have
\begin{equation} \label{narrow-broad-decompose}
    \|E_\Phi f\|_{L^q(X)}^q \lesssim A^{O(1)} \sum_\si \|E_\Phi f_\si\|_{L^q(X)}^q + K^{O(1)}\|\BR_A E_\Phi f\|_{L^q(X)}^q.
\end{equation}
If the second term is dominant, the desired estimate follows directly
from Corollary \ref{l2-lq-broad-estimate}. Indeed, under \eqref{t-full-range} together with \eqref{pq-range-final}, we have $p\geq 2$. We may therefore assume that
the first term dominates.

Fix $\si$, and let $\xi_\si$ denote the center of $\si$. By a dyadic
pigeonholing argument (see, for instance, \cite[Proposition 3.2]{Wu}),
we may assume that there exist
\[
\mu\lesssim K^{\min\{2t,t+1\}}
\]
and a finitely overlapping collection $\mathcal{B}$ of
$K\times K^2$ rectangles whose major axes are parallel to
$(\Phi'(\xi_\si),-1)$, such that for every $B\in\mathcal{B}$,
\begin{enumerate}
    \item the quantities $\|E_\Phi f_\si\|_{L^q(B)}$ are comparable;
    \item
    \[
    |X\cap B|\sim \mu.
    \]
\end{enumerate}
Moreover,
\begin{equation}\label{after-pigeonholing}
    \|E_\Phi f_\si\|_{L^q(X)}^q
    \lessapprox
    \mu K^{-3}
    \|E_\Phi f_\si\|_{L^q(\bigcup \mathcal{B})}^q.
\end{equation}
If $\#\mathcal{B}<K^{100}$, the trivial estimate
$|E_\Phi f_\si|\lesssim K^{-1+1/p}\|f_\si\|_p$
gives the desired bound with a loss $K^{O(1)}$, which is absorbed into
$R^\e$ after choosing the small parameters. We may therefore assume
that $\#\mathcal{B}\geq K^{100}$.
For any $r\in[1,R/K]$ and any $rK\times rK^2$ rectangle
$B_{rK\times rK^2}$ whose major axis is parallel to
$(\Phi'(\xi_\si),-1)$, the Katz--Tao $(R^{-1},t)$ condition on the
$R^{-1}$-dilate of $X$ implies that
\begin{equation}\label{Katz--Tao-pre-condition}
\#\{B\in\mathcal{B}: B\cap B_{rK\times rK^2}\neq\varnothing\}
\lesssim
\mu^{-1} r^t K^{\min\{2t,1+t\}}.
\end{equation}
For $(x_1,x_2)\in\mathbb{R}^2$, define the linear map
$\mathcal{L}_\si:\mathbb{R}^2\to\mathbb{R}^2$ by
\[
\mathcal{L}_\si(x_1,x_2)
:=
\left(
\frac{x_1+\Phi'(\xi_\si)x_2}{K},
\frac{x_2}{K^2}
\right).
\]

By rescaling, \eqref{Katz--Tao-pre-condition}, and
Lemma \ref{katz-tao-set-lem}, we may select a subcollection
$\mathcal{B}'\subset\mathcal{B}$ such that
\begin{equation}\label{Katz--Tao-ed-B}
    \#\mathcal{B}'
    \approx
    \mu K^{-\min\{2t,1+t\}}\#\mathcal{B}.
\end{equation}
Set
\[
Y_\si
:=
\mathcal{L}_\si\left(\bigcup_{B\in\mathcal{B}'}B\right).
\]
Moreover, after localization to any ball of radius $R/K^2$, the
$(R/K^2)^{-1}$-dilate of $Y_\si$ is a Katz--Tao
$\big((R/K^2)^{-1},t\big)$-set.

We claim that the $(R/K)^{-1}$-dilate of $Y_\si$ is a Katz--Tao
$\big((R/K)^{-1},t\big)$-set. In view of the preceding local
Katz--Tao estimate, it remains to show that, for every
\[
\frac{R}{K^2}\leq r\leq \frac{R}{K},
\]
the number of unit balls contained in $Y_\si$ and intersecting any
ball of radius $r$ is $\lesssim r^t$.

Suppose first that $t\geq 1$. Since $Y_\si$ is contained in a rectangle
of dimensions $R/K^2\times R/K$, for every
$r\in[R/K^2,R/K]$, the intersection of $Y_\si$ with any ball of radius
$r$ can be covered by
\[
\lesssim \frac{r}{R/K^2}
\]
balls of radius $R/K^2$. Applying the local Katz--Tao estimate on each
such ball, we obtain
\[
\#\{\text{unit balls of $Y_\si$ intersecting } B_r\}
\lesssim
\left(\frac{R}{K^2}\right)^t
\frac{r}{R/K^2}
=
r\left(\frac{R}{K^2}\right)^{t-1}
\leq r^t,
\]
where the last inequality follows from $r\geq R/K^2$.

Now suppose that $0\leq t<1$. In this case, we simply use the total
number of rectangles in $\mathcal{B}'$. By \eqref{Katz--Tao-ed-B},
\[
\#\mathcal{B}'
\lesssim
\mu K^{-2t}\#\mathcal{B}
\lesssim
K^{-2t}R^t
=
\left(\frac{R}{K^2}\right)^t
\leq r^t.
\]
This proves the claim.

We now perform the standard parabolic rescaling in order to apply induction on scales. Define
\[
g_\si(\xi)
:=
f_\si\left(\xi_\si+K^{-1}\xi\right)
\]
and
\[
\Phi_\si(\xi)
:=
K^2\left[
\Phi\left(\xi_\si+K^{-1}\xi\right)
-\Phi(\xi_\si)
-\Phi'(\xi_\si)K^{-1}\xi
\right].
\]
By \eqref{narrow-broad-decompose}, \eqref{after-pigeonholing}, and \eqref{Katz--Tao-ed-B}, we obtain
\[
\|E_\Phi f\|_{L^q(X)}^q
\lessapprox A^{O(1)} K^{\min\{2t,t+1\}-q}
\sum_\si \|E_{\Phi_\si}g_\si\|_{L^q(Y_\si)}^q.
\]
Since $\Phi_\si$ also satisfies \eqref{phi-condition}, induction on scale gives
\[\|E_\Phi f\|_{L^q(X)}^q \lesssim_\e (R/K)^{q\e} A^{O(1)} K^{\min\{2t,t+1\}-q+\frac{q}{p}} \sum_\si \|f_\si\|_p^q.\]
Under the assumptions \eqref{pq-range-final}, we have $q\geq p$ and
\[
\min\{2t,t+1\}-q+\frac{q}{p} \leq 0.
\]
Therefore, the induction closes.
\end{proof}
Determining the smallest exponent $\alpha\geq 0$ for which the estimate
\[
\|\BR_A E_\Phi f\|_{L^2(X)}
\lesssim_\e
R^\e |X|^\alpha \|f\|_2
\]
holds for all $f$ and all unions $X$ of unit balls whose $R^{-1}$-dilate
is a Katz--Tao $(R^{-1},t)$-set is a difficult problem.
In \cite{Wu}, an example was constructed showing that
$\alpha\geq 1/6$ when $t=1$. In this paper, we extend this construction
to the range $1/2\leq t\leq 2$.

\begin{theorem}\label{broad-lower-bound-example}
Let $\Phi(\xi)=\xi^2$ and fix a sufficiently small $\e>0$.
With $K\sim R^{\e^4}$ and $A\sim R^{\e^{100}}$ as above,
for every $1/2\leq t\leq2$ and all sufficiently large $R$, there exist
$f\in\mathcal{S}(\mathbb{R})$ with
$\operatorname{supp} f\subset[-1,1]$ and a union $X\subset B_R(0)$
of unit balls such that the $R^{-1}$-dilate of $X$ is a Katz--Tao
$(R^{-1},t)$-set and
\[
\|\BR_A E_\Phi f\|_{L^2(X)}
\gtrsim_\e
R^{-\e}|X|^{\frac{2t-1}{6t}}\|f\|_2.
\]
\end{theorem}

The proof is given in Section \ref{proof-example}.
\bigskip

\section{\texorpdfstring{$L^2$}{L2} broad estimates}\label{proof-main-result}
In this section, we establish the $L^2$ broad estimates. Throughout this
section, whenever we say that a dilate of a set is a Katz--Tao set, we
allow an arbitrary translation of the dilated set. Our main result is
the following.

\begin{theorem}
Let $0\leq t\leq 2$, $1\leq r\leq R$, and suppose that
$A\sim r^{\e^{100}}$. Let $X\subset B_R(0)$ be a union of unit balls,
and assume that the $r^{-1}$-dilate of $X$ is a Katz--Tao
$(r^{-1},t)$-set. Then, for every
$f\in\mathcal{S}(\mathbb{R})$ with
$\operatorname{supp}f\subset[-1,1]$ and every
$\Phi\in C^\infty([-1,1];\mathbb{R})$ satisfying
\eqref{phi-condition}, we have
\[
\|\BR_A E_\Phi f\|_{L^2(X)}
\lesssim_\e
R^\e r^\e |X|^{\frac{2t}{3t+6}}\|f\|_2,
\qquad 0\leq t<1,
\]
and
\[
\|\BR_A E_\Phi f\|_{L^2(X)}
\lesssim_\e
R^\e r^\e |X|^{\frac{4t-2}{9t}}\|f\|_2,
\qquad 1\leq t\leq2.
\]
\end{theorem}
We now begin the proof. By a standard two-ends reduction
(see, for example, \cite{Wu,Wang-Wu24}), we may assume that there exist
parameters $\lambda,\beta>0$, a collection $\mathbb{T}$ of
$r$-scale wave packets, a parameter $A'\gtrapprox A$, a finitely
overlapping collection $\mathbb{B}$ of balls of radius
$r^{1-\e^2}$, each intersecting $X$, and a shading $Y$ such that, for
each $T\in\mathbb{T}$, $Y(T)$ is a collection of
$r^{1-\e^2}\times r^{1/2+\de_0}$ tubes $J\subset T$ satisfying the
following properties:
\begin{enumerate}
    \item For every $T\in\mathbb{T}$, the quantities $\|f_T\|_2$ are comparable, and
    \[
    \#Y(T)\sim\beta.
    \]

    \item For every $T\in\mathbb{T}$ and every $J\in Y(T)$, the tube $J$
    is a union of $\sim\lambda$ finitely overlapping balls of radius
    $r^{1/2+\de_0}$.

    \item We have
    \begin{equation} \label{two-ends-decomposition}
    \|\BR_A E_\Phi f\|_{L^2(X)}^2
    \lessapprox
    \sum_{B_k \in \mathbb{B}}
    \left\|
    \BR_{A'} E_\Phi
    f_k
    \right\|_{L^2(X_k)}^2,
    \end{equation}
    where \[\mathbb{T}_k := \{T\in\mathbb{T},\bigcup_{J\in Y(T)} J \cap B_k \cap X \neq \emptyset\}\]
    \[f_k := \sum_{T\in\mathbb{T}_k}f_T, \qquad X_k:= B_k \cap X \cap
    \bigcup_{T\in\mathbb{T}_k}\bigcup_{J\in Y(T)} J.\]
\end{enumerate}
\subsection{Non-two-ends scenario}
If $\beta\leq r^{\e^4}$, we apply induction on scales. Since
\[
A'\geq r^{(1-\e^2)\e^{100}}
\]
and the $r^{\e^2-1}$-dilate of $X_k$ is a Katz--Tao
$(r^{\e^2-1},t)$-set, the induction hypothesis gives
\[
\left\|
    \BR_{A'} E_\Phi f_k
\right\|_{L^2(X_k)}^2
\lesssim_\e
R^{2\e} r^{2\e-2\e^3}
|X|^{\frac{4t}{3t+6}}
\|f_k\|_2^2,
\qquad 0\leq t<1,
\]
and
\[
\left\|
    \BR_{A'} E_\Phi f_k
\right\|_{L^2(X_k)}^2
\lesssim_\e
R^{2\e} r^{2\e-2\e^3}
|X|^{\frac{8t-4}{9t}}
\|f_k\|_2^2,
\qquad 1\leq t\leq2.
\]

Combining these estimates with \eqref{two-ends-decomposition}, we obtain
\begin{equation}\label{indu-scale-small-t}
\left\|
    \BR_A E_\Phi f
\right\|_{L^2(X)}^2
\lessapprox_\e
R^{2\e} r^{2\e-2\e^3}
|X|^{\frac{4t}{3t+6}}
\sum_{B_k\in\mathbb{B}}\|f_k\|_2^2,
\qquad 0\leq t<1,
\end{equation}
and
\begin{equation}\label{indu-scale-large-t}
\left\|
    \BR_A E_\Phi f
\right\|_{L^2(X)}^2
\lessapprox_\e
R^{2\e} r^{2\e-2\e^3}
|X|^{\frac{8t-4}{9t}}
\sum_{B_k\in\mathbb{B}}\|f_k\|_2^2,
\qquad 1\leq t\leq2.
\end{equation}

Since $\beta\leq r^{\e^4}$, each wave packet $T$ belongs to at most
$\lesssim r^{\e^4}$ of the collections $\mathbb{T}_k$. Hence, by
Lemma \ref{l2-orthogonality},
\[
\sum_{B_k\in\mathbb{B}}\|f_k\|_2^2
\lesssim
r^{\e^4}\|f\|_2^2.
\]
Substituting this into \eqref{indu-scale-small-t} and
\eqref{indu-scale-large-t} closes the induction. 

\medskip

\subsection{Two-ends scenario}
If $\beta>r^{\e^4}$, we make use of three estimates. The first was established
in \cite[equation (2.50)]{Wu} via the hairbrush argument.

\begin{proposition}\label{hair-brush-estimate}
If $\beta>r^{\e^4}$, then
\begin{equation} \label{estimate-0}
    \|\BR_A E_\Phi f\|_{L^2(X)}^2
\lesssim
r^{O(\e^2)}
|X|r^{-1/2}\|f\|_2^2.
\end{equation}
\end{proposition}
We next turn to the remaining two estimates. By a dyadic pigeonholing
argument, we may assume that there exist parameters $s,M>0$, a collection
$\mathcal{Q}$ of $r^{1/2}$-cubes, a collection $\mathcal{T}$ of
$r$-scale tubes, and an integer $\bar{A}\gtrapprox A$ with the following
properties.

For each $Q\in\mathcal{Q}$, define
\[
\mathcal{T}(Q)
:=
\{T\in\mathcal{T}:T\cap Q\neq\varnothing\},
\]
and, for each $K^{-1}$-cap $\si$, define
\[
\cT_\si(Q)
:=
\left\{
T\in\cT(Q):
c_{\theta_T}\in\si
\right\}.
\]
Also set
\[
f_Q
:=
\sum_{T\in\mathcal{T}(Q)}f_T.
\]

After pigeonholing, we may further assume that:
\begin{enumerate}
    \item For every $Q\in\mathcal{Q}$,
    \[
    |X\cap Q|\sim s.
    \]

    \item For every $Q\in\mathcal{Q}$ and every $K^{-1}$-cap $\si$,
    \[
    \#\mathcal{T}_\si(Q)=0
    \quad\text{or}\quad
    \#\mathcal{T}_\si(Q)\sim M,
    \]
    while
    \[
    \#\mathcal{T}(Q)\gtrsim \bar{A}M.
    \]

    \item The quantities
    \[
    \|\BR_{\bar{A}}E_\Phi f_Q\|_{L^2(X\cap Q)}
    \]
    are comparable as $Q$ ranges over $\mathcal{Q}$.

    \item
    \[
    \mathcal{T}
    =
    \bigcup_{Q\in\mathcal{Q}}\mathcal{T}(Q).
    \]

    \item The quantities $\|f_T\|_2$ are comparable for
    $T\in\mathcal{T}$.

    \item We have
    \[
    \|\BR_A E_\Phi f\|_{L^2(X)}^2
    \lessapprox
    \sum_{Q\in\mathcal{Q}}
    \|\BR_{\bar{A}}E_\Phi f_Q\|_{L^2(X\cap Q)}^2.
    \]
\end{enumerate}
For any ball $B$ of radius $r_0\geq r^{1/2}$, the number of cubes
$Q\in\mathcal{Q}$ intersecting $B$ is bounded by
\[
\lesssim
\frac{r_0^t}{s}
=
\left(\frac{r_0}{r^{1/2}}\right)^t
\frac{r^{t/2}}{s}.
\]
Hence, by Lemma \ref{katz-tao-set-lem}, there exists a subcollection
$\mathcal{Q}'\subset\mathcal{Q}$ such that
\begin{equation}\label{subcollection-Q}
\#\mathcal{Q}'
\gtrapprox
s r^{-t/2}\#\mathcal{Q},
\end{equation}
and the $r^{-1}$-dilate of $\bigcup \mathcal{Q}'$ is a Katz--Tao $(r^{-1/2},t)$-set.

Define
\[
\mathcal{T}'
:=
\bigcup_{Q\in\mathcal{Q}'}
\mathcal{T}(Q).
\]
Applying Corollary \ref{two-ends-furstenberg-estimate-dual} to \(\mathcal Q'\), together with \eqref{subcollection-Q}, we obtain
\begin{equation}\label{card-T-prime}
\#\mathcal{T}'
\gtrapprox_{\e}
K^{-O(1)} s r^{-3t/4}\#\cq
\begin{cases}
M^{1+t/2},
& 0\leq t<1,\\[4pt]
M^{2-t/2},
& 1\leq t\leq2.
\end{cases}
\end{equation}
\smallskip

\subsubsection{A standard $L^2$ estimate}

We first recall a standard $L^2$ estimate established in \cite{Wu}.

\begin{proposition}
We have
\[
\|\BR_A E_\Phi f\|_{L^2(X)}^2
\lesssim
K M^2 s r^{-1/2}
(\#\cq)(\#\mathcal{T}')^{-1}
\|f\|_2^2.
\]
\end{proposition}

Combining this estimate with \eqref{card-T-prime}, we obtain
\begin{equation}\label{estimate-1-small-t}
\|\BR_A E_\Phi f\|_{L^2(X)}^2
\lessapprox_\e
K^{O(1)}
M^{1-\frac{t}{2}}
r^{\frac{3t}{4}-\frac{1}{2}}
\|f\|_2^2,
\qquad
0\leq t<1,
\end{equation}
and
\begin{equation}\label{estimate-1-large-t}
\|\BR_A E_\Phi f\|_{L^2(X)}^2
\lessapprox_\e
K^{O(1)}
M^{\frac{t}{2}}
r^{\frac{3t}{4}-\frac{1}{2}}
\|f\|_2^2,
\qquad
1\leq t\leq2.
\end{equation}
\smallskip

\subsubsection{Application of refined decoupling}

By applying Theorem \ref{refined-decoupling-thm}, Wu \cite{Wu} obtained
the following estimate.
Here refined decoupling is applied on the retained set
$X\cap\bigcup\mathcal{Q}$, whose measure is comparable to
$s\,\#\mathcal{Q}$.

\begin{proposition}
We have
\[
\|\BR_A E_\Phi f\|_{L^2(X)}^2
\lessapprox
K^{O(1)}
(s\,\#\mathcal{Q})^{2/3}
M^{2/3}
(\#\mathcal{T}')^{-2/3}
\|f\|_2^2.
\]
\end{proposition}

Combining this estimate with \eqref{card-T-prime}, we obtain
\begin{equation}\label{estimate-2-small-t}
\|\BR_A E_\Phi f\|_{L^2(X)}^2
\lessapprox_\e
K^{O(1)}
M^{-\frac{t}{3}}
r^{\frac{t}{2}}
\|f\|_2^2,
\qquad
0\leq t<1,
\end{equation}
and
\begin{equation}\label{estimate-2-large-t}
\|\BR_A E_\Phi f\|_{L^2(X)}^2
\lessapprox_\e
K^{O(1)}
M^{\frac{t-2}{3}}
r^{\frac{t}{2}}
\|f\|_2^2,
\qquad
1\leq t\leq2.
\end{equation}
\smallskip
\subsubsection{Interpolation}
Interpolating between \eqref{estimate-0},
\eqref{estimate-1-small-t}, \eqref{estimate-1-large-t},
\eqref{estimate-2-small-t}, and \eqref{estimate-2-large-t},
we obtain the desired estimate and complete the proof.
\bigskip
\section{\texorpdfstring{$L^p$}{Lp}-circular decay of Fourier transforms of fractal measures} \label{proof-decay}
In this section, we prove Theorem \ref{decay}. Our argument follows the
strategy developed in \cite{Wu}. The main modification is to apply
Theorem \ref{main-result} in place of the corresponding weighted
restriction estimate for $t=1$, which leads to the new decay bounds in
the full range of $t$ considered here.

\begin{proof} [Proof of Theorem \ref{decay}]
Set
\[
m(t):=\min\{2t,t+1\},
\qquad
q_0(t):=
\max\left\{
\frac{6t+12}{6-t},\frac{18t}{t+4}
\right\}.
\]
Fix $1\leq p\leq2$, and let $p'$ be its conjugate exponent.
Choose
\[
q:=\max\{q_0(t),m(t)p\},
\qquad
u:=\min\{p',q\}.
\]
Since
\[
q_0(t)\geq\max\{2,m(t)+1\},
\]
we have $2\leq u\leq p'$ and
\[
q\geq
\max\left\{
q_0(t),m(t)\frac{u}{u-1},u
\right\}.
\]
Thus the pair $(u,q)$ satisfies the hypotheses of
Theorem \ref{main-result}.

Let $f\in L^{p'}(\ZS^1,d\sigma)$ satisfy
$\|f\|_{L^{p'}(\ZS^1,d\sigma)}=1$.
Following the dyadic pigeonholing and Katz--Tao extraction argument
in \cite{Wu}, using Lemma \ref{katz-tao-set-lem}, we obtain a union
$Y\subset B_R(0)$ of unit balls whose $R^{-1}$-dilate is a
Katz--Tao $(R^{-1},t)$-set and such that
\[
\left|\int E_Sf(Rx)\,d\mu(x)\right|
\lessapprox
R^{-t}\int_Y|E_Sf(x)|\,dx.
\]
Since $|Y|\lesssim R^t$, H\"older's inequality and
Theorem \ref{main-result} give
\[
\begin{aligned}
\left|\int E_Sf(Rx)\,d\mu(x)\right|
&\lessapprox
R^{-t}|Y|^{1-1/q}\|E_Sf\|_{L^q(Y)}\\
&\lesssim_{\e,p}
R^{-t+\e}|Y|^{1-1/q}
\|f\|_{L^u(\ZS^1,d\sigma)}\\
&\lesssim_{\e,p}
R^{-t/q+\e},
\end{aligned}
\]
where we used $u\leq p'$ in the last step.
The conclusion follows by duality, since
\[
\frac{t}{q}=\beta_t(p).
\]
\end{proof}

\bigskip

\section{Further discussions}\label{proof-example}
In this section, we prove Theorem \ref{broad-lower-bound-example}.
We refine the arithmetic construction in \cite{Wu} by controlling the
contribution from each $K^{-1}$-interval separately. A slight reduction
of the denominator range makes the argument uniform up to $t=2$.

We first record an elementary weighted Gauss-sum estimate. As above,
$e(z)=e^{2\pi i z}$.

\begin{lemma}\label{weighted-quadratic-gauss}
Let $q\geq1$ be odd, let $a,b\in\mathbb{Z}$ satisfy $(b,q)=1$, and set
\[
G(a,b;q):=\sum_{r=0}^{q-1}e\left(\frac{br^2+ar}{q}\right).
\]
Then $|G(a,b;q)|=q^{1/2}$. Moreover, if $I$ is a finite interval of integers,
$N\geq1$, and $W\in C^1(J;\mathbb{C})$, where $J$ is an interval containing $\{k/N:k\in I\}$,
then
\begin{equation}\label{weighted-gauss-expansion}
\begin{aligned}
\sum_{k\in I}e\left(\frac{bk^2+ak}{q}\right)W(k/N)
={}&\frac{G(a,b;q)}{q}\sum_{k\in I}W(k/N)\\
&+O\left(q^{1/2}\log(2q)
\left(\|W\|_{L^\infty(J)}+\int_J|W'(s)|\,ds\right)\right).
\end{aligned}
\end{equation}
\end{lemma}
\begin{proof}
The case $q=1$ is immediate, so assume that $q\geq3$.
Since $2b$ is invertible modulo $q$,
\[
|G(a,b;q)|^2
=\sum_{h\bmod q}e\left(\frac{bh^2+ah}{q}\right)
\sum_{r\bmod q}e\left(\frac{2bhr}{q}\right)=q.
\]
Write $\omega_k=e((bk^2+ak)/q)$ and $m=G(a,b;q)/q$.
The sequence $\omega_k-m$ is $q$-periodic with mean zero.
After deleting complete periods, Fourier expansion on
$\mathbb{Z}/q\mathbb{Z}$ gives, for every integer interval $I$,
\[
\left|\sum_{k\in I}(\omega_k-m)\right|
\leq q^{-1/2}\sum_{j=1}^{q-1}
\min\left\{q,\frac{C}{\|j/q\|_{\mathbb{R}/\mathbb{Z}}}\right\}\lesssim q^{1/2}\log(2q).
\]
Here every nonzero Fourier coefficient before division by $q$ is
another quadratic Gauss sum of modulus $q^{1/2}$.
The notation $\|\cdot\|_{\mathbb{R}/\mathbb{Z}}$ denotes distance to
the nearest integer.
Summation by parts proves \eqref{weighted-gauss-expansion}.
\end{proof}

\begin{proof}[Proof of Theorem \ref{broad-lower-bound-example}]
Fix a sufficiently small $\e>0$ and take $R$ sufficiently large.
Replacing $R$ by a comparable smaller square if necessary, we may
assume that $R=N^2$ with $N\in\mathbb{N}$. Set
\[
Q:=\left\lfloor\min\left\{
R^{\frac{2t-1}{6}},\frac{N}{K^3}\right\}\right\rfloor.
\]
Since $K\sim R^{\e^4}$ and $1/2\leq t\leq2$, we have
\begin{equation}\label{example-Q-range}
1\leq Q\leq R^{\frac{2t-1}{6}},\qquad
Q\leq\frac{N}{K^3},\qquad
Q\gtrsim K^{-3}R^{\frac{2t-1}{6}}.
\end{equation}
Choose a nonnegative function
$\psi\in C_c^\infty((-10^{-3},10^{-3}))$ with
$\psi\equiv1$ near the origin, and define
\[
f(\xi):=\sum_{|k|\leq N/2}
\psi\left(R\left(\xi-\frac{k}{N}\right)\right).
\]
The bumps have disjoint supports, and $\operatorname{supp}f\subset[-1,1]$.
Consequently,
\begin{equation}\label{example-f-norm}
\|f\|_2\sim (N/R)^{1/2}=N^{-1/2}.
\end{equation}

\smallskip\noindent
\emph{The arithmetic set and its Katz--Tao condition.}
Define
\[
\mathcal{P}:=\left\{
\left(N\left(n+\frac aq\right),R\frac bq\right):
\begin{array}{l}
q\in[Q,2Q)\cap(2\mathbb{Z}+1),\\
a,b,n\in\mathbb{Z},\quad 0\leq a<q,\\
0\leq b<q/2,\quad (b,q)=1,\quad |n|\leq N/4
\end{array}\right\}.
\]
Let $X$ be the union of the unit balls centered at $\mathcal{P}$.
These balls are contained in $B_R(0)$ for large $R$.
The condition $(b,q)=1$ makes $b/q$ a reduced fraction, so a horizontal
row determines $b$ and $q$ uniquely. Different rows have vertical
separation at least $R/(4Q^2)$, and the points in one row have
horizontal separation at least $N/(2Q)$. By
\eqref{example-Q-range}, the unit balls are therefore disjoint.
For $Q\geq2$, the elementary estimate
\[
\sum_{\substack{Q\leq q<2Q\\q\text{ odd}}}q\varphi(q)\sim Q^3
\]
gives the number of centers; here $\varphi$ is Euler's totient function.
This estimate follows directly from M\"obius inversion and partial
summation. If $Q=1$, the only arithmetic parameters are
$q=1$ and $a=b=0$, giving $\sim N$ centers. In both cases,
\begin{equation}\label{example-X-cardinality}
|X|\sim\#\mathcal{P}\sim NQ^3\lesssim R^t.
\end{equation}

For any $1\leq r\leq R$ and any $z\in\mathbb{R}^2$, the separation
between the rows and the spacing within each row imply
\begin{equation}\label{KT-example-count}
\#(\mathcal{P}\cap B_r(z))
\lesssim
\left(1+\frac{rQ^2}{R}\right)
\left(1+\frac{rQ}{N}\right)\lesssim 1+\frac{rQ}{N}+\frac{r^2Q^3}{N^3}.
\end{equation}
The last term satisfies
\[
\frac{r^2Q^3}{N^3}
=\left(\frac rR\right)^2NQ^3
\lesssim\left(\frac rR\right)^tR^t=r^t.
\]
If $t\geq1$, then $rQ/N\leq r\leq r^t$.
If $1/2\leq t\leq1$, then
\[
\frac{rQ}{N}
=\frac rR\,NQ
\leq\left(\frac rR\right)^tNQ^3
\lesssim r^t.
\]
Thus \eqref{KT-example-count} is $\lesssim r^t$ at every relevant
scale. The same bound holds for the number of unit balls of $X$
meeting $B_r(z)$, by enlarging the radius by a fixed factor.
Hence the $R^{-1}$-dilate of $X$ is a Katz--Tao $(R^{-1},t)$-set.

\smallskip\noindent
\emph{A lower bound on every interior frequency interval.}
Fix a center
$p=(N(n+a/q),Rb/q)\in\mathcal{P}$ and write $x=p+u$ with $|u|\leq1$.
For a bump wholly contained in a $K^{-1}$-interval $\si$, the
substitution $\xi=k/N+v/R$ gives its exact contribution as
\[
R^{-1}e\left(\frac{bk^2+ak}{q}\right)W_x(k/N),
\]
where
\[
W_x(s):=e(u_1s+u_2s^2)
\int\psi(v)e\left(
\left(\frac{x_1}{R}+\frac{2x_2s}{R}\right)v
+\frac{x_2v^2}{R^2}\right)\,dv.
\]
Here we used $e(nk)=1$ and $R=N^2$.
Take $\si\subset[-1/4,1/4]$, and let $I_\si$ be the consecutive
integers $k$ for which the support interval of the $k$th bump is
contained in $\si$. Then
\[
\#I_\si=N/K+O(1).
\]
At most two additional bumps meet the endpoints of $\si$, so
\[
E_\Phi f_\si(x)
=R^{-1}\sum_{k\in I_\si}
e\left(\frac{bk^2+ak}{q}\right)W_x(k/N)
+O(R^{-1}\|\psi\|_1).
\]

Let $c_\si$ be the center of $\si$.
On this interval the variation of $u_1s+u_2s^2$ from its value at
$c_\si$ is $O(K^{-1})$. The phase in the $v$-integral is uniformly
small because $|x|\leq R$ and $|v|\leq10^{-3}$.
Therefore, for large $K$,
\[
\operatorname{Re}\left(
e(-u_1c_\si-u_2c_\si^2)W_x(s)\right)
\geq\frac12\int\psi(v)\,dv,
\qquad s\in\si.
\]
Moreover,
\[
\|W_x\|_{L^\infty(\si)}
+\|W_x'\|_{L^\infty(\si)}\lesssim\|\psi\|_1.
\]
Applying Lemma \ref{weighted-quadratic-gauss}, we obtain
\[
|E_\Phi f_\si(x)|
\geq c_\psi\frac{N}{RK\sqrt q}
-C_\psi R^{-1}\big(\sqrt q\log(2q)+1\big).
\]
Since $q<2Q\leq2N/K^3$,
\[
\frac{q\log(2q)}{N/K}
\lesssim\frac{\log R}{K^2}\longrightarrow0.
\]
The boundary-bump error is also negligible. Thus, uniformly in
$x\in X$ and $\si\subset[-1/4,1/4]$,
\begin{equation}\label{example-pointwise-lower}
|E_\Phi f_\si(x)|\gtrsim\frac{1}{NK\sqrt Q}.
\end{equation}
The fixed finitely overlapping cover contains $\gtrsim K$ intervals
inside $[-1/4,1/4]$. Since $A\sim R^{\e^{100}}\ll K$, the definition
of the broad part and \eqref{example-pointwise-lower} give
\[
\BR_A E_\Phi f(x)\gtrsim\frac{1}{NK\sqrt Q},
\qquad x\in X.
\]

\smallskip\noindent
\emph{The exponent in the lower bound.}
By \eqref{example-X-cardinality} and \eqref{example-f-norm},
\[
\|\BR_A E_\Phi f\|_{L^2(X)}
\gtrsim\frac{(NQ^3)^{1/2}}{NK\sqrt Q}
\sim\frac QK\|f\|_2.
\]
Using \eqref{example-Q-range}, we have
\[
\frac{Q}{(NQ^3)^{\frac{2t-1}{6t}}}
=\left(\frac{Q}{R^{\frac{2t-1}{6}}}\right)^{\frac1{2t}}
\gtrsim K^{-\frac3{2t}}.
\]
Since $1/2\leq t\leq2$, it follows that
\[
\begin{aligned}
\|\BR_A E_\Phi f\|_{L^2(X)}
&\gtrsim K^{-1-\frac3{2t}}
|X|^{\frac{2t-1}{6t}}\|f\|_2\\
&\gtrsim K^{-4}|X|^{\frac{2t-1}{6t}}\|f\|_2\\
&\gtrsim_\e R^{-\e}|X|^{\frac{2t-1}{6t}}\|f\|_2,
\end{aligned}
\]
where the last step uses $K\sim R^{\e^4}$ and sufficiently small
$\e$. This proves the theorem, including $t=1/2$ and $t=2$.
\end{proof}
\bigskip

\section{Funding Declaration}
This work was supported by the National Key R\&D Program of China
[grant number 2022YFA1005700]. Z. Cao was supported by the High-level
Talent Research Start-up Project Funding of Henan Academy of Sciences
[grant number 241819032] and the National Natural Science Foundation
of China [grant numbers 12401117, 12371100]. J. Li is supported by the
National Natural Science Foundation of China [grant number 12471090]
and the Fundamental Research Funds for the Central Universities.
X. Wang acknowledges financial support from the Department of Mathematics,
University of Illinois Urbana-Champaign.


\begin{thebibliography}{99}
\bibitem{BG}
J.~Bourgain and L.~Guth.
\newblock Bounds on oscillatory integral operators based on multilinear estimates.
\newblock {\em Geom. Funct. Anal.}, 21(6):1239--1295, 2011.

\bibitem{DGOWWZ}
X.~Du, L.~Guth, Y.~Ou, H.~Wang, B.~Wilson, and R.~Zhang.
\newblock Weighted restriction estimates and application to Falconer distance set problem.
\newblock {\em Amer. J. Math.}, 143(1):175--211, 2021.

\bibitem{Erdogan}
M.~B.~Erdo\u{g}an.
\newblock A bilinear Fourier extension theorem and applications to the distance set problem.
\newblock {\em Int. Math. Res. Not.}, 2005(23):1411--1425, 2005.

\bibitem{G}
L.~Guth.
\newblock A restriction estimate using polynomial partitioning.
\newblock {\em J. Amer. Math. Soc.}, 29(2):371--413, 2016.

\bibitem{GIOW}
L.~Guth, A.~Iosevich, Y.~Ou, and H.~Wang.
\newblock On {F}alconer's distance set problem in the plane.
\newblock {\em Invent. Math.}, 219(3):779--830, 2020.

\bibitem{LucaRogers}
R.~Luc\`a and K.~M.~Rogers.
\newblock Average decay of the Fourier transform of measures with applications.
\newblock {\em J. Eur. Math. Soc.}, 21(2):465--506, 2019.

\bibitem{Shayya}
B.~Shayya.
\newblock Weighted restriction estimates using polynomial partitioning.
\newblock {\em Proc. Lond. Math. Soc.}, 115(3):545--598, 2017.

\bibitem{Wolff}
T.~Wolff.
\newblock Decay of circular means of Fourier transforms of measures.
\newblock {\em Int. Math. Res. Not.}, 1999(10):547--567, 1999.

\bibitem{Wu}
S.~Wu. 
\newblock Weighted $L^2$ estimates with applications to $L^p$ problems.
\newblock arXiv:2506.02650, 2025.

\bibitem{Wang-Wu24}
H.~Wang and S.~Wu.
\newblock Restriction estimates using decoupling theorems and two-ends {F}urstenberg inequalities.
\newblock arXiv:2411.08871, 2024.

\bibitem{Wang-Wu26}
H.~Wang and S.~Wu.
\newblock Two-ends Furstenberg estimates in the plane.
\newblock {\em Math. Ann.}, 395(4), Paper No. 91, 31 pp., 2026.

\end{thebibliography}
\end{document}